\documentclass[11pt,reqno]{article}
\usepackage{latexsym, amsmath, amssymb, amsthm, a4, epsfig}
\usepackage{graphicx}
\usepackage{color}

\newtheorem{theorem}{Theorem}[section]

\newtheorem{lemma}[theorem]{Lemma}

\newcommand{\ds}{\displaystyle}
\newcommand{\p}{\partial}

\newcommand{\rank}{\text{rank}\,}

\newcommand{\Rbb}{\mathbb{R}}

\newcommand{\la}{\langle}
\newcommand{\ra}{\rangle}

\newcommand{\Acal}{\mathcal{A}}
\newcommand{\Bcal}{\mathcal{B}}

\newcommand{\Kcal}{\mathcal{K}}
\newcommand{\Mcal}{\mathcal{M}}

\newcommand{\Scal}{\mathcal{S}}

\newcommand{\Ga}{\alpha}

\newcommand{\Gd}{\delta}
\newcommand{\Ge}{\epsilon}

\newcommand{\Gvf}{\varphi}

\newcommand{\Gl}{\lambda}

\newcommand{\Gt}{\theta}

\newcommand{\Gr}{\rho}

\newcommand{\GG}{\Gamma}
\newcommand{\GL}{\Lambda}

\newcommand{\GO}{\Omega}

\newcommand{\beq}{\begin{equation}}
\newcommand{\eeq}{\end{equation}}

\def\ol{\overline}

\numberwithin{equation}{section}
\numberwithin{figure}{section}

\begin{document}

\title{Generic properties of the Neumann-Poincar\'e operator: simplicity of eigenvalues and cyclic vectors\thanks{\footnotesize KA and YM were partially supported by JSPS of Japan KAKENHI grants 20K03655 and 21K13805, HK by NRF (of S. Korea) grant 2021R1A2B5B02-001786, and MP by a Simons Collaboration Grant for Mathematicians.}}

\author{Kazunori Ando\thanks{Department of Electrical and Electronic Engineering and Computer Science, Ehime University, Ehime 790-8577, Japan. Email: {\tt ando@cs.ehime-u.ac.jp}.}
\and Hyeonbae Kang\thanks{Department of Mathematics and Institute of Applied Mathematics, Inha University, Incheon 22212, S. Korea. Email: {\tt hbkang@inha.ac.kr}.}
\and Yoshihisa Miyanishi\thanks{Department of Mathematical Sciences, Faculty of Science, Shinshu University, A519, Asahi 3-1-1, Matsumoto 390-8621, Japan. Email: {\tt miyanishi@shinshu-u.ac.jp}}
\and \and Mihai Putinar\thanks{University of California at Santa Barbara, CA 93106, USA and
Newcastle University, Newcastle upon Tyne NE1 7RU, UK, {\tt mputinar@math.ucsb.edu, mihai.putinar@ncl.ac.uk}}}

\date{}
\maketitle

\begin{abstract}
Two generic properties of the Neumann--Poincar\'e operator are investigated. We prove that non-zero eigenvalues of the Neumann--Poincar\'e operator on smooth boundaries in three dimensions and higher are generically simple in the sense of Baire category. We also prove that the functions defined by the fundamental solutions to the Laplace operator located at points outside the surface are generically cyclic vectors in the sense that the collection of those points where the functions are not cyclic vectors is of measure zero.
\end{abstract}



\section{Introduction}
Let $\Omega$ be a bounded domain in $\Rbb^d$ ($d \ge 2$) with the $C^{1, \Ga}$-smooth boundary for some $\Ga>0$. The Neumann-Poincar\'e (abbreviated to NP) operator  $\Kcal_{\p \GO}^*$ on $H^{-1/2}(\p \GO)$ (the $L^2$-Sobolev space of order $-1/2$) is defined by
\beq
{\Kcal}_{\p \GO}^*[\psi](x) := \int_{\p \GO} \p_{n_x} \GG(x, y) \psi(y) \; dS_y, \quad x \in \p\GO,
\eeq
where $\p_{n_x}$ denotes the outer normal derivative (with respect to $x$) on $\p \GO$, $\GG$ the fundamental solution to the Laplacian, namely,
$$
\GG(x, y)
=
\begin{cases}
\ds \frac{1}{2\pi} \log {|x-y|}, \quad &d=2, \\
\ds \frac{-1}{d(d-2) c_d} \frac{1}{|x-y|^{d-2}}, \quad &d \geq 3,
\end{cases}
$$
$c_d$ the volume of the unit ball in $\Rbb^d$, and $dS_y$ the surface element on $\p \GO$.

Let $\Scal_{\p\GO}$ be the single layer potential on $\p\GO$, namely,
\beq
{\Scal}_{\p \GO}[\psi](x) := \int_{\p \GO} \GG(x, y) \psi(y) \; dS_y, \quad x \in \Rbb^d.
\eeq
If we consider ${\Scal}_{\p \GO}[\psi](x)$ for $x \in \p\GO$, then ${\Scal}_{\p \GO}$ maps $H^{-1/2}(\p \GO)$ into $H^{1/2}(\p \GO)$. As such an operator, ${\Scal}_{\p \GO}$ is invertible if $d \ge 3$. If $d=2$, it may not be invertible. But, if it is so, we may dilate $\GO$ and the single layer potential on the dilated surface is invertible. Since the NP spectrum (spectrum of the NP operator) is invariant under dilation, we may assume from the beginning that ${\Scal}_{\p \GO}: H^{-1/2}(\p \GO) \to H^{1/2}(\p \GO)$ is invertible. Let $\la \cdot , \cdot \ra$ denote the $H^{-1/2}$-$H^{1/2}$ pairing on $\p\GO$. Then the bilinear form
\beq \label{eq:inner_product}
\la \Gvf, \psi \ra_{\p\GO} :=-\la \Gvf, {\Scal}_{\p \GO}[\psi] \ra
\eeq
on $H^{-1/2}(\p \GO)$ becomes an inner product, and ${\Kcal}_{\p \GO}^*$ is self-adjoint on $H^{-1/2}(\p \GO)$ equipped with this inner product.

If $\p\GO$ is $C^{1,\Ga}$ as we assume in this paper, the NP operator $\Kcal_{\p\GO}^*$ is compact on $H^{-1/2}(\p \GO)$, and hence $\Kcal_{\p\GO}^*$ admits sequences of nonzero eigenvalues of finite multiplicities converging to $0$ while $0$ may or may not be an eigenvalue. These NP eigenvalues (eigenvalues of the NP operator) may have multiplicities higher than 1. For example, $0$ is the only NP eigenvalue (other than $1/2$) on disks. It is shown in \cite{KPS} (see also \cite{JK}) that $0$ is an NP eigenvalue of infinite multiplicities on lemniscates. If $\GO$ is a ball in $\Rbb^3$, then NP eigenvalues are $\frac{1}{2(2k+1)}$
with the multiplicity $2k+1$ ($k=0, 1, 2, \ldots$). In this paper, we show that the hypersurface with non-simple nonzero NP eigenvalues is rare by showing that nonzero NP eigenvalues are generically simple. Here, the notion `generic' means the collection $E$ with such a property in a topological space $X$ contains an
intersection of at most countable number of open dense subsets in $X$. The terminology is in general
used when $X$ is a Baire space: in such spaces the set $E$ is dense and said to be Baire typical. There are many works on generic properties of eigenvalues and eigenfunctions (see \cite{Uhlen} and references therein). We also mention the article \cite{Bando} which motivates the study of this paper. There the generic simplicity of eigenvalues of the Laplace-Beltrami operator on compact Riemannian manifolds is proved.

To define the class of boundary surfaces, let $D$ be a domain in $\Rbb^d$ ($d \ge 3$) with the smooth boundary $\p D$.  Fix constants $L>0$ and $c>0$ ($c$ to be chosen), and define
\begin{align}
\Mcal := & \{ G=(g_1, g_2, \ldots, g_d) \mid G \text{ is smooth in a neighborhood of } \p D,  \nonumber \\
& \quad \| G - I \|_{C^1(\p D)} \le c, \ \ |G(x)-G(y)| > L |x-y| \text{ for all } x \neq y \in \p D \},
\end{align}
where $I$ stands for the identity mapping. We choose $c$ so small that $G(\p D)$ is the boundary of a bounded domain for each $G \in \Mcal$ which we denote by $D_G$, namely, $\p D_G=G(\p D)$. We introduce the metric on $\Mcal$:
$$
p_n(G) = \left( \sum_{i=1}^d \max_{x \in \p D, \, |\Ga| \le n} |\p^\Ga g_i(x)|^2 \right)^{1/2}
$$
for $n = 0, 1, \dots$, where $\Ga$ is multi-indices.  The distance on $\Mcal$ is defined by
$$
\rho(F,G) = \sum_n \frac{1}{2^n}\frac{p_n(F-G)}{1+ p_n(F-G)}.
$$
Note that $\ol{\Mcal}$ is the collection of all smooth functions $G=(g_1, g_2, \ldots, g_d)$ satisfying
$$
|G(x)-G(y)| \ge L |x-y|
$$
for all $x,y \in \p D$ and $(\ol{\Mcal}, \rho)$ is a complete metric space.

We have the following theorem regarding generic simplicity of the NP eigenvalues:

\begin{theorem}\label{thm: generic simlicity}
The set $\Mcal_s$ of all $G \in \Mcal$ such that all non-zero NP eigenvalues on $\p D_G$ are simple is Baire typical in $(\ol{\Mcal}, \rho)$.
\end{theorem}

The second generic property is regarding cyclic vectors. Let $K$ be a bounded operator on a separable Hilbert space $H$. For $\xi \in H$, we denote by $L_\xi$ the cyclic subspace generated by $\xi$, namely,
\begin{equation}
L_\xi= \overline{\text{span} \{ K^n \xi \mid n=0,1, 2, \ldots \}}.
\end{equation}
The vector $\xi$ is called a cyclic vector for $K$ if
\begin{equation}\label{fullone}
L_{\xi} = H.
\end{equation}
We also consider $N$-cyclic vectors. The $N$-tuple $(\xi_1, \xi_2, \ldots, \xi_N)$ of vectors in $H$ is called $N$-cyclic vectors for if
\begin{equation}\label{fulltwo}
L_{\xi_1} + L_{\xi_2} + \cdots + L_{\xi_N} = H.
\end{equation}
It is worth mentioning that $E=(\xi_1, \xi_2, \ldots, \xi_N)$ is called $N$-supercyclic vectors for $K$ if $\cup_{n=0}^\infty K^n(E)$ is dense in $H$.

In relation to $N$-cyclic vectors for the NP operator $\Kcal_{\p\GO}^*$ on $\p\GO$, we seek a class of explicit functions on $\p\GO$ which provides $N$-cyclic vectors generically in some sense, under the assumption that the spectral multiplicity of $\Kcal_{\p\GO}^*$ is finite. The spectral multiplicity is the maximal multiplicity of the eigenvalues (see \cite[Section 51]{Halmos}). As candidates of such functions, we consider
\begin{equation}\label{qz}
q_z(x):= v \cdot \nabla_z \Gamma(z-x), \quad x \in \p\GO
\end{equation}
for $z \in \Rbb^d \setminus \ol{\GO}$ where $v$ is a constant vector. Actually functions $q_z$ were used to extract spectral information of the NP operator on polygons numerically in \cite{HKL}.

We obtain the following theorem. 

\begin{theorem}\label{cyclic2}
Let $\GO$ be a bounded domain in $\Rbb^d$ ($d \ge 2$) with the $C^{1,\Ga}$ boundary for some $\Ga>0$. Suppose that the spectral multiplicity of $\Kcal_{\p\GO}^*$ on $H^{-1/2}(\p\GO)$ is $N < +\infty$.  Then $(q_{z_1}, q_{z_2}, \ldots, q_{z_N})$ are $N$-cyclic vectors for $\Kcal_{\p\GO}^*$ for almost all $(z_1, z_2, \ldots, z_N) \in (\Rbb^3 \setminus \ol{\GO})^N$, namely,
\begin{equation}\label{full3}
L_{q_{z_1}} + L_{q_{z_2}} + \cdots + L_{q_{z_N}} = H^{-1/2}(\p\GO).
\end{equation}
\end{theorem}

We emphasize that any $(q_{z_1}, q_{z_2}, \ldots, q_{z_M})$ with $M <N$ cannot generate $H^{1/2}(\p\GO)$ (see Lemma \ref{lem:M<N}). If all eigenvalues are simple, $q_z$ is a cyclic vector for $\Kcal_{\p\GO}^*$ for almost all $z \in \Rbb^d \setminus \ol{\GO}$.

We also obtain the following theorem with the same proof.

\begin{theorem}\label{cyclic3}
Let $\GO$ be a bounded domain in $\Rbb^d$ ($d \ge 2$) with the $C^{1,\Ga}$ boundary for some $\Ga>0$. Let $\Gvf_n$, $n=1,2, \ldots$, be an orthogonal system of eigenfunctions of $\Kcal_{\p\GO}^*$. For almost all $z \in \Rbb^3 \setminus \ol{\GO}$, $q_z$ satisfies the following property:
\beq\label{eigenmode}
\la q_z, \Gvf_n \ra_{\p\GO} \neq 0 \quad \text{for all $n$}.
\eeq
\end{theorem}

The above theorem shows that for almost all $z$, $q_z$ contains all eigenmodes. The theorem remains true for any system of functions $\phi$ belonging to a {\it countable} subset
of $H^{-1/2}(\GO)$.

The rest of the paper is devoted to proving Theorem \ref{thm: generic simlicity} (section \ref{sec:two}) and Theorems \ref{cyclic2} and \ref{cyclic3} (section \ref{sec:three}). This paper ends with a short discussion.

\section{Proof of Theorem \ref{thm: generic simlicity}}\label{sec:two}

We restrict ourselves to $\mathbb{R}^3$ even though the argument can be applied to $\mathbb{R}^d$, $d \ge 3$.

For a bounded domain $\GO$ with the smooth boundary $\p\GO$, we define the perturbation $\GO(ha)$ for a smooth function $a$ with $\max_{x \in \p\GO} |a(x)| \le 1$ and a small real parameter $h$ as
\begin{equation}\label{normal_pert}
  \p \GO(ha) = \{ x + h a(x) n(x) \, ; \, x \in \p \GO \}.
\end{equation}

The proof of this section relies heavily on the following Hadamard variational type formula obtained in \cite{Grieser}. We also mention \cite[section 136]{RN} and references therein for general theory for analytic branches of eigenvalues and eigenvectors for perturbation of a self-adjoint operator on a Hilbert space.

\begin{theorem}[\cite{Grieser}]\label{thm:G}
  Let $\GO$ be a smooth bounded domain in $\mathbb{R}^3$, and let $\Gl \neq 0, 1/2$ be an eigenvalue of the NP operator $\Kcal^*$ on $\p \GO$ with the eigenspace $E$.
  Then there are $h_0 > 0$ and real analytic functions $h \mapsto \Gl^{(i)}(h)$, $h \to e^{(i)}(h)$ defined for $|h| < h_0$, $i = 1, \dots, \dim E$, such that
  \begin{itemize}
  \item[{\rm (i)}] $\Gl^{(i)}(0)=\Gl$ for all $i$ and $\{ e^{(1)}(0), \dots, e^{(\dim E)}(0) \}$ is a basis of $E$,
  \item[{\rm (ii)}] for each $h$ the numbers $\Gl^{(i)}(h)$ are eigenvalues of $\Kcal_h^*$ (the NP operator on $\p \GO(ha)$) with eigenfunctions $e^{(i)}(h)$ satisfying
  \beq\label{eq:normalization}
  \la e^{(i)}(h), e^{(i)}(h) \ra_{\p\GO(ha)} = 1.
  \eeq
  \end{itemize}
For each analytic branch $\Gl^{(i)}(h)$ and $e^{(i)}(h)$, we have
  \begin{equation} \label{Hadamard_variation}
    \frac{d \Gl^{(i)}}{dh} (0) = \int_{\p \GO} a \left[ |\nabla_\p \Scal[e^{(i)}(0)]|^2 - c(\Gl) (\p_n \Scal[e^{(i)}(0)] |_-)^2 \right] d\sigma,
  \end{equation}
where $\Scal$ is the single layer potential on $\p \GO$, the subscript $-$ indicates the limit (to $\p \GO$) from $\GO$,
$\nabla_\p u^{(i)} =\nabla u^{(i)} - (\p_n u^{(i)}) n$ on $\p \GO$, and
$$
c(\Gl)= \frac{1+2\Gl}{1-2\Gl}.
$$
\end{theorem}

Some remarks on Theorem \ref{thm:G} are in order. The normalization \eqref{eq:normalization} differs from \cite{Grieser} by the constant factor and \eqref{Hadamard_variation} differs accordingly. There the normalization is given by $\| \nabla \Scal_h[e^{(i)}(h)] \|_{L^2(\p\GO(h a))}=1$. Using divergence theorem and the jump formula for the normal derivative of the single layer potential, namely,
\beq\label{singlejump}
\p_n \Scal_{h}[\Gvf] |_{\pm} = \left( \pm \frac{1}{2} I + \Kcal_{h}^* \right)[\Gvf] \quad\text{on } \p\GO,
\eeq
one can see
$$
\| \nabla \Scal_h[e^{(i)}(h)] \|_{L^2(\p\GO(h a))}^2 = (1/2 - \Gl^{(i)}(h)) \la e^{(i)}(h), e^{(i)}(h) \ra_{\p\GO(ha)}.
$$
The function $e^{(i)}(h)$ is analytic in the sense that there is a sequence of functions $\psi_0^{(i)}, \psi_1^{(i)}, \psi_2^{(i)}, \ldots$ in $H^{-1/2}(\p\GO)$ such that
$$
e^{(i)}(h)= \psi_0^{(i)}+ h \psi_1^{(i)} + h^2 \psi_2^{(i)} + \cdots.
$$
which converges in $H^{-1/2}(\p\GO)$.

Note that the perturbation \eqref{normal_pert} is in the direction of the normal vectors while the perturbation in the class $\Mcal$ which is under consideration in this paper is given in terms of diffeomorphisms. So, in order to be able to apply Theorem \ref{thm:G} to the context of this paper, we need to show that perturbations by diffeomorphisms can be realized locally as perturbations of the form \eqref{normal_pert} (and vice versa). The following lemma does it. 

\begin{lemma}\label{lem:ah}
For each $F \in \Mcal$, there is a constant $C>0$ such that for any sufficiently small $\Ge>0$ the following hold:
\begin{itemize}
\item[{\rm (i)}] If $\Gr(F,G) < \Ge$, then there is $a \in C^\infty(\p D_F)$ with $\max_{x \in \p D_F} |a(x)| \le 1$ such that
\beq\label{eq:aep}
\p D_G =  \p D_F(C\Ge a) .
\eeq
\item[{\rm (ii)}] If a bounded domain $\GO$ with the smooth boundary satisfies $\p\GO= \p D_F(\Ge a)$ for some $a \in C^\infty(\p D_F)$ with $\max_{x \in \p D_F} |a(x)| \le 1$, then there is $G \in \Mcal$ with $\Gr(F,G) < C\Ge$ such that
\beq
\p\GO=\p D_G.
\eeq
\end{itemize}
\end{lemma}

\begin{proof}
(i) Fix $F \in \Mcal$ and suppose that $G \in \Mcal$ satisfies $\Gr(F,G) < \Ge$. Let $H:=G \circ F^{-1}$. Then, $H (\p D_F) = \p D_G$. Since $|H(x)-x| < C\Ge$ for some $C$ for all $x \in \p D_F$, we have
\beq\label{eq:dist}
d (y, \p D_F) < C\Ge \quad \forall y \in \p D_G.
\eeq

We claim that for each $x \in \p D_F$, there is unique $t$ such that $x + t n(x) \in \p D_G$. Moreover, $t$ satisfies $|t| < C\Ge$. Once the claim is proved, we denote such a $t$ by $C\Ge a(x)$. Then it is easy to prove, using implicit function theorem, that $a(x)$ is a smooth function. So, \eqref{eq:aep} is proved.

To prove the claim, let $x_0 \in \p D_F$. After rotation and translation if necessary, we may assume that $x_0=(0,0,0)$ and there is an open neighborhood $U \subset \p D_F$ of $x_0$ which is represented as a graph of a smooth function $f$ satisfying $f(0,0)=0$ and $\nabla f(0,0)=0$, namely, $U=\{(u, f(u)) \mid u \in B_r \}$ for some $r>0$. Here $B_r=\{ u \in \Rbb^2 \mid |u|<r \}$. In particular, $n(x_0)=(0,0,1)$. Then $H(U) \subset \p D_G$. Let $\pi$ be the projection $\pi(x_1,x_2,x_3)=(x_1,x_2)$. Since $|H(x)-x| < C\Ge$, we have
$$
B_{r-C\Ge} \subset \pi (H(U)) \subset B_{r+C\Ge}.
$$
Thus, if $\Ge$ is sufficiently small, then $(0,0) \in \pi (H(U))$, which implies that $(0,0,t) \in H(U) \subset \p D_G$ for some $t$. Because of \eqref{eq:dist}, $|t| < C\Ge$.

To prove uniqueness of such $t$, let $\eta_F$ be a defining function for $\p D_F$, namely, $\eta_F$ is smooth in a neighborhood $V$ of $\p D_F$, $\p D_F = \eta_F^{-1}(0)$ and $\min_{x \in \p D_F} |\nabla \eta_F(x)| \ge C_1$ for some $C_1>0$. Because of \eqref{eq:dist}, we may assume that $\p D_G \subset V$. Since $n(x_0)=(0,0,1)$, we have $|\frac{\p\eta_F}{\p x_3}(x_0)| \ge C_1$. Thus there is $\Gd>0$ such that
\beq\label{eq:tdelta}
\left| \frac{\p\eta_F}{\p x_3}(0,0,s) \right| \ge \frac{C_1}{2} \quad\text{if } |s| < \Gd.
\eeq
Note that $\eta_G:= \eta_F \circ H^{-1}$ is a defining function for $\p D_G$. Suppose that there are two different $t_1, t_2$ with $|t_j| < C \Ge$ such that $(0,0, t_j) \in \p D_G$. Then
$$
0= \eta_G(0,0,t_1) - \eta_G(0,0,t_2) = \frac{\p\eta_G}{\p x_3}(0,0,t_*)(t_1-t_2)
$$
for some $t_*$. Thus, $\frac{\p\eta_G}{\p x_3}(0,0,t_*)=0$, and hence $|\frac{\p\eta_F}{\p x_3}(0,0,t_*)| < C\Ge$. Thus, if $\Ge$ is sufficiently small, then this contradicts \eqref{eq:tdelta}.

(ii) Suppose $\p\GO= \p D_F(\Ge a)$. If $\Ge$ is sufficiently small, the mapping $x \mapsto x+\Ge a(x)n(x)$ is bijection from $\p D_F$ onto $\p\GO$, which can be extended to a tubular neighborhood of $\p D_F$ as an injective mapping. In fact, the mapping $x + t n(x) \mapsto x+(\Ge a(x)+t)n(x)$, $|t| <\Gd$ for some $\Gd>0$, is such an extension. Denote the extended mapping by $Q$. Then, there is a constant $C>0$ such that
$$
|Q(x)-Q(y)| > (1-C\Ge) |x-y|
$$
for all $x,y \in \p D_F$. Let $G:=Q \circ F$. Then there is $L_F >L$ such that
$$
|G(x)-G(y)| > (1-C\Ge) |F(x)-F(y)| > (1-C\Ge) L_F |x-y|
$$
for all $x,y \in \p D$. Thus, if $\Ge$ is sufficiently small, then
$$
|G(x)-G(y)| > L |x-y|
$$
for every $x \neq y \in \p D$, which implies $G \in \Mcal$. If $x \in \p D$, then $G(x)-F(x)= \Ge a(F(x)) n(F(x))$, and hence $\Gr(F, G) < C \Ge$. This completes the proof.
\end{proof}

The following lemma will be used to prove Lemma \ref{eq:number}. We emphasize that estimates in the lemma are not optimal.

\begin{lemma}\label{lem:uniform}
Let $\GO$ be a bounded domain in $\Rbb^d$ ($d \ge 3$) with the smooth boundary and let $a$ be a smooth function on $\p\GO$ such that $\max_{x \in \p D_F} |a(x)| \le 1$. Let $\Psi(x):= x + \Ge a(x) n(x)$ for $x \in \p\GO$. Denoting by $\Kcal^*$ and $\Kcal_\Ge^*$ the NP operators on $\p\GO, \ \p\GO(\Ge a)$, respectively, and by $\Scal, \ \Scal_\Ge$ the single layer potentials, the following hold: there is a constant $C>0$ such that
\beq\label{eq:stablenp}
\| \Kcal^* [\Gvf] - \Kcal_\Ge^* [\Gvf \circ \Psi^{-1}] \circ \Psi \|_{L^2(\p \GO)} \le C\Ge \| \Gvf \|_{L^2(\p \GO)}
\eeq
and
\beq\label{eq:stablesingle}
\| \Scal [\Gvf] - \Scal_\Ge [\Gvf \circ \Psi^{-1}] \circ \Psi \|_{H^{1/2}(\p \GO)} \le C\Ge \| \Gvf \|_{L^2(\p \GO)}.
\eeq
for all sufficiently small $\Ge>0$ and $\Gvf \in L^2(\p \GO)$.
\end{lemma}

\begin{proof}
We assume that $d=3$ since the case for higher dimensions can be dealt with in the same way. Let $K(x,y)$ ($x,y\in \p\GO$) be the integral kernel for $\Kcal$. Fix $x_0 \in \p\GO$. Using the local coordinates, we may assume that $x_0=0$ and $\p\GO$ near $x_0$ is given as a graph, namely $(u, f(u))$ where $u \in B_r$ and $f$ is a smooth function in $B_r$. If $x=(u,f(u))$ and $y=(v,f(v))$ lie on $\p\GO$, then by Taylor's theorem we have
\begin{align}
K(x,y) &= \frac{1}{4\pi \sqrt{1+|\nabla f(u)|^2}} \frac{f(u)-f(v)-\nabla f(u) (u-v)}{(|u-v|^2 + |f(u)-f(v)|^2)^{3/2}} \nonumber \\
&= \frac{1}{4\pi \sqrt{1+|\nabla f(u)|^2}} \frac{(u-v) \cdot R(u,v)(u-v)}{|u-v|^3 (1+ E(u,v))^{3/2}}, \label{eq:Kxy}
\end{align}
where $R(u,v)=(R_{ij}(u,v))$ with
\beq
R_{ij}(u,v) = \int_0^1 (1-t) \p_i\p_j f((1-t)u+tv) dt
\eeq
and
\beq
E(u,v)= \frac{|f(u)-f(v)|^2}{|u-v|^2}.
\eeq

If $x=(u,f(u))$, then we have
\beq\label{eq:TuFu}
\Psi(x) = (T(u), F(u))
\eeq
where
\beq
T(u) = u - \frac{\Ge \tilde{a}(u) \nabla f(u)}{\sqrt{1+|\nabla f(u)|^2}}
\eeq
and
\beq
F(u) = f(u) + \frac{\Ge \tilde{a}(u)}{\sqrt{1+|\nabla f(u)|^2}}.
\eeq
Here $\tilde{a}(u):= a(u,f(u))$. Thus there is $r_1 \le r$ such that $\p\GO(\Ge a)$ near $\Psi(x_0)$ is given as a graph $(u,f_\Ge(u))$, $u \in B_{r_1}$, where
$$
f_\Ge(u)= F(T^{-1}(u)).
$$
Let $K_\Ge(x,y)$ be the integral kernel of $\Kcal_\Ge$. If $x'=(u,f_\Ge(u))$ and $y'=(v,f_\Ge(v))$ lie on $\p\GO(\Ge a)$, then $K_\Ge(x',y')$ is given by \eqref{eq:Kxy} with $f$ replaced by $f_\Ge$.

Let $x=(u,f(u)), \ y=(v,f(v)) \in \p\GO$ and $x'= \Psi(x)$ and $y'= \Psi(y)$. Then, by \eqref{eq:TuFu}, we have
\begin{align*}
K_\Ge(\Psi(x), \Psi(y))= \frac{1}{4\pi \sqrt{1+|\nabla f_\Ge(T(u))|^2}} \frac{(T(u)-T(v)) \cdot R^\Ge(u,v)(T(u)-T(v))}{|T(u)-T(v)|^3 (1+ E_\Ge(u,v))^{3/2}}, \label{eq:Kxy}
\end{align*}
where $R^\Ge (u,v)=(R_{ij}^\Ge (u,v))$ with
\beq
R_{ij}^\Ge (u,v) = \int_0^1 (1-t) \p_i\p_j f_\Ge((1-t)T(u)+tT(v)) dt
\eeq
and
\beq
E_\Ge(u,v)= \frac{|f_\Ge(T(u))-f_\Ge(T(v))|^2}{|T(u)-T(v)|^2}= \frac{|F(u)-F(v)|^2}{|T(u)-T(v)|^2}.
\eeq
Since
$$
T(u) = u + \Ge T_1(u), \quad F(u) = f(u) + \Ge a_1(u), \quad f_\Ge(u) = f(u) + \Ge a_2(u)
$$
where $T_1, a_1, a_2$ are smooth functions in $B_{r_1}$ bounded independently of $\Ge$ ($T_1$ is a $\Rbb^2$-valued function), it is straight-forward to see that
$$
K_\Ge(\Psi(x), \Psi(y))= K(x,y) + \Ge A(x,y),
$$
where $A(x,y)$ satisfies
\beq
|A(x,y)| \le \frac{C}{|x-y|}
\eeq
for some constant $C$. Let $\Acal$ be the integral operator defined by $A(x,y)$. It is easy to see that $\Acal$ is bounded on $L^2(\p \GO)$.

The Jacobian determinant $|J\Psi(x)|$ of $\Psi$ on $\p\GO$ is of the form $|J\Psi(x)|=1+\Ge b(x)$ for some bounded smooth function $b$. Let $\Gvf \in L^2(\p \GO)$. We have
\begin{align*}
\Kcal_\Ge^* [\Gvf \circ \Psi^{-1}] (\Psi(x)) = \int_{\p\GO} K_\Ge(\Psi(x), \Psi(y)) \Gvf(y)(1+\Ge b(y)) dS_y ,
\end{align*}
and hence
\begin{align*}
\Kcal_\Ge^* [\Gvf \circ \Psi^{-1}] (\Psi(x)) - \Kcal^* [\Gvf](x)
= \Ge \Acal[\Gvf(1+\Ge b)](x) + \Ge \Kcal^*[\Gvf b](x).
\end{align*}
Thus \eqref{eq:stablenp} follows.

Note that
\begin{align*}
\Scal_\Ge^* [\Gvf \circ \Psi^{-1}] (\Psi(x)) = \frac{1}{4\pi} \int_{\p\GO} \frac{\Gvf(y)(1+\Ge b(y))}{|\Psi(x)-\Psi(y)|} dS_y .
\end{align*}
It is easy to see that
\beq
\frac{1}{|\Psi(x)-\Psi(y)|} = \frac{1}{|x-y|} + \Ge B(x,y),
\eeq
where $B(x,y)$ satisfies
\beq\label{eq:est1}
|B(x,y)| \le \frac{C}{|x-y|}
\eeq
and
\beq\label{eq:est2}
|B(x,y)-B(x',y)| \le \frac{C|x-x'|}{|x-y|^2}
\eeq
for some constant $C$ and for all $x,x',y \in \p\GO$. Let $\Bcal$ be the integral operator defined by $B(x,y)$. Then, we have
$$
\Scal_\Ge^* [\Gvf \circ \Psi^{-1}] (\Psi(x)) - \Scal^* [\Gvf](x)
= \Ge \Bcal[\Gvf(1+\Ge b)](x) + \Ge \Scal^*[\Gvf b](x).
$$
It is known that an integral operator whose kernel satisfies \eqref{eq:est1} and \eqref{eq:est2} is bounded from $L^2(\p\GO)$ into $H^s(\p\GO)$ for any $s<1$ (see, for example, \cite[Theorem A.1]{FKM}). So, in particular, \eqref{eq:stablesingle} follows.
\end{proof}

For $F \in \Mcal$, we denote by $\Kcal_F^*$ and $\Scal_F$ the NP operator and the single layer potential on $\p D_F$.
Let, for ease of notation,
\beq
\la \Gvf, \psi \ra_F:= \la \Gvf, \psi \ra_{\p D_F}
\eeq
for $\Gvf, \psi \in H^{-1/2}(\p D_F)$, and let
\beq
\| \Gvf \|_F^2 := \la \Gvf, \Gvf \ra_F.
\eeq

The following lemma shows that the number of NP eigenvalues in an interval away from $0$ is invariant under small perturbation in $\Mcal$. 

\begin{lemma}\label{eq:number}
Let $F \in \Mcal$ and $a,b \in (0,1/2)$ be numbers such that $-a,b$ are not eigenvalues of $\Kcal_F^*$. Let $I$ be either $(-1/2,-a)$ or $(b,1/2)$. There is $\Ge>0$ such that if $\Gr(F,G)< \Ge$, then
\beq\label{2000}
\#(I,F)=\#(I,G)
\eeq
where $\#(I,F)$ denotes the number of eigenvalues of $\Kcal_F^*$ in $I$ counting multiplicities.
\end{lemma}

\begin{proof}
We only prove lemma for the case when $I=(b,1/2)$. Suppose that $\#(I,F)=N$, and let $\Gl_i$, $1 \le i \le N$, be the eigenvalues of $\Kcal_F^*$ in $I$. Suppose that $\Gr(F,G)< \Ge$ for a sufficiently small $\Ge$. Then, by Lemma \ref{lem:ah} (i), there is $a \in C^\infty(\p D_F)$ with $\max_{x \in \p D_F} |a(x)| \le 1$ such that $\p D_G =  \p D_F(\Ge_1 a)$ ($\Ge_1=C\Ge$). Let $h_0$ be the number appearing in Theorem \ref{thm:G}. If $\Ge_1 < h_0$, then there are real analytic functions $\Gl^{(i)}(h)$, $e^{(i)}(h)$ ($1 \le i \le N$) defined for $|h| < h_0$ such that $\Gl^{(i)}(0)=\Gl_i$, $\Gvf_i:=e^{(i)}(0)$ is the corresponding eigenfunction, and $\Gl^{(i)}(\Ge_1)$ is an eigenvalue of $\Kcal_G^*$ with the corresponding eigenfunction $e^{(i)}(\Ge_1)$ where all the eigenfunctions are normalized by \eqref{eq:normalization}. Thus $\#(I,G) \ge N$. We note that each $\Gvf_i$ belongs to $L^2(\p\GO)$. In fact, since $\Kcal_F^*$ is bounded from $H^{-1/2}(\p\GO)$ into $L^2(\p\GO)$, $\Gvf_i= (\Gl_i^F)^{-1} \Kcal_F^*[\Gvf_i] \in L^2(\p\GO)$.

Suppose $\#(I,G) > N$. Then there is $\psi \in H^{-1/2}(\p D_G)$ such that $\| \psi \|_G=1$, $\Kcal_G^*[\psi]= \Gl \psi$ for some $\Gl \in I$, and
$$
\la e^{(i)}(\Ge_1), \psi \ra_G=0, \quad 1 \le i \le N.
$$
As before, we have $\psi \in L^2(\p D_G)$. Let $\Psi(x)= x+ \Ge_1 a(x) n(x)$ for $x \in \p D_F$. Then $\Psi$ is a diffeomorphism between $\p D_F$ and $\p D_G$ and can be extended to a tubular neighborhood of $\p D_F$ as a diffeomorphism as explained in the proof of Theorem \ref{thm:G}. We then have
\beq\label{eq:1000}
|\la e^{(i)}(\Ge_1) \circ \Psi, \psi \circ \Psi \ra_F| \le C\Ge, \quad 1 \le i \le N.
\eeq
In fact, since $\psi \in L^2(\p D_G)$, we have
\begin{align*}
\la e^{(i)}(\Ge_1) \circ \Psi, \psi \circ \Psi \ra_F &= -\la e^{(i)}(\Ge_1) \circ \Psi, \Scal_F [\psi \circ \Psi] \ra \\
&= -\la e^{(i)}(\Ge_1) \circ \Psi, \Scal_G [\psi] \circ \Psi \ra + O(\Ge) \\
&= \la e^{(i)}(\Ge_1) , \psi \ra_G + O(\Ge) = O(\Ge),
\end{align*}
where the second equality holds thanks to \eqref{eq:stablesingle} and the third equality holds since $|J\Psi| - 1 =O(\Ge)$ where $|J\Psi|$ is the Jacobian of $\Psi$.

Since $h \mapsto e^{(i)}(h)$ is real analytic, we have $\|e^{(i)}(\Ge_1) \circ \Psi-\Gvf_i\|_F \le C \Ge$ for some $C>0$ for all $i$. Thus we infer using \eqref{eq:1000} that
$$
\la \Gvf_i, \psi \circ \Psi \ra_F = \la \Gvf_i - e^{(i)}(\Ge_1) \circ \Psi, \psi \circ \Psi \ra_F + \la e^{(i)}(\Ge_1) \circ \Psi, \psi \circ \Psi \ra_F = O(\Ge), \quad 1 \le i \le N.
$$

Let
$$
\Gvf:= \psi \circ \Psi - \sum_{i=1}^N \la \Gvf_i, \psi \circ \Psi \ra_F \Gvf_i.
$$
Then $\la \Gvf_i, \Gvf \ra_F =0$ for $i=1,2, \ldots, N$. Moreover, we have
\begin{align*}
\|\Gvf\|_F^2 &= \la \psi \circ \Psi, \psi \circ \Psi \ra_F - \sum_{i=1}^N \la \Gvf_i, \psi \circ \Psi \ra_F^2 \\
&= \la \psi , \psi \ra_G + O(\Ge^2) \\
&= 1 + O(\Ge^2).
\end{align*}
Thanks to \eqref{eq:stablenp}, we also have
\begin{align*}
\la \Gvf, \Kcal_F^*[\Gvf] \ra_F &= \la \psi \circ \Psi, \Kcal_F^*[\psi \circ \Psi] \ra_F - \sum_{i=1}^N \Gl_i^F \la \Gvf_i, \psi \circ \Psi \ra_F^2 \\
&= \la \psi , \Kcal_G^*[\psi] \ra_G + O(\Ge^2) \\
&= \Gl + O(\Ge^2).
\end{align*}
Thus, we have
$$
\frac{|\la \Gvf, \Kcal_F^*[\Gvf] \ra_F|}{\|\Gvf\|_F^2}= \frac{|\Gl| + O(\Ge^2)}{1 + O(\Ge^2)} \in I
$$
if $\Ge$ is sufficiently small.

So far, we proved that $E:= \{\Gvf_1, \Gvf_2, \ldots, \Gvf_N, \Gvf/\|\Gvf\|_F \}$ is an orthonormal set such that for each $f \in E$, we have $\la f, \Kcal_F^*[f] \ra_F \in I$. This implies via the Min-Max Principle that there are at least $N+1$ eigenvalues of $\Kcal_F^*$ in $I$. This contradicts the assumption that $\#(I,F)=N$. Thus \eqref{2000} holds.
\end{proof}

Let $F \in \Mcal$.
For a nonzero eigenvalue $\Gl$ of $\Kcal_F^*$ and the corresponding normalized eigenfunction $\Gvf$, let
\beq
I_\Gvf (x)= |\nabla_\p \Scal_F[\Gvf](x)|^2 - c(\Gl) (\p_n \Scal_F[\Gvf] |_- (x))^2, \quad x \in \p D_F.
\eeq
Note that $I_\Gvf$ is the function appearing in the formula \eqref{Hadamard_variation}.

We will also use the following lemma.

\begin{lemma}\label{eq:not=}
Let $F \in \Mcal$. If $\Gl$ is a nonzero, non-simple eigenvalue of $\Kcal_F^*$, then there are two normalized eigenfunctions $\Gvf_1, \Gvf_2$ corresponding to $\Gl$ such that
\beq\label{eq:not=2}
I_{\Gvf_1} (x) \neq I_{\Gvf_2} (x) \quad \text{for some $x \in \p D_F$}.
\eeq
\end{lemma}

\begin{proof}
Let $\Gvf_1, \Gvf_2$ be orthonormal eigenfunctions corresponding to the eigenvalue $\Gl$ and assume that \eqref{eq:not=2} does not hold, that is, $I_{\Gvf_1} (x) = I_{\Gvf_2} (x)$ for all $x \in \p D_F$. Then, we have
\begin{align}
  & \big( \nabla_\p u_1 (x) - \nabla_\p u_2 (x) \big) \cdot \big(\nabla_\p u_1 (x) + \nabla_\p u_2 (x) \big) \nonumber \\
  & \quad = c(\Gl) \big(\p_n u_1 |_-(x) - \p_n u_2 |_-(x) \big) \big(\p_n u_1 |_-(x) + \p_n u_2 |_- (x) \big), \label{eq:id}
\end{align}
where $u_j=\Scal_F[\Gvf_j]$ ($j=1,2$).

Let $x_0$ be the maximum point of $u_1-u_2$ on $\ol{D_F}$. Since $u_1-u_2$ is a harmonic function, $x_0 \in \p D_F$ by the maximum principle (see \cite{GT}) and hence
\begin{equation}\label{1=2}
 \nabla_\p (u_1-u_2) (x_0) =  0.
\end{equation}
However, by Hopf's lemma, we have
\begin{equation*}
  0 < \p_n (u_1-u_2)|_{-} (x_0) .
\end{equation*}
We then infer from \eqref{eq:id} that
\beq\label{eq:1=-2}
\p_n u_1 |_-(x_0) + \p_n u_2|_-(x_0)  = 0.
\eeq
In particular, we have
\beq\label{eq:nonzero}
\p_n u_1 |_-(x_0) \neq 0.
\eeq

We now consider the following function:
\begin{align} \label{integrand_w_theta}
 \GL(\Gt) := I_{(\cos{\Gt}) \Gvf_1 + (\sin{\Gt}) \Gvf_2}(x_0), \quad \Gt \in \Rbb.
\end{align}
We show that $\GL$ is not constant.
Suppose on the contrary that $\GL$ is constant. Then, we have
\begin{align*}
 \frac{d\GL}{d\Gt}
  &= 2 \big( - \sin{\Gt} \nabla_\p u_1(x_0) + \cos{\Gt} \nabla_\p u_2(x_0) \big) \cdot \big( \cos{\Gt} \nabla_\p u_1(x_0) + \sin{\Gt} \nabla_\p u_2(x_0) \big) \\
  & \quad + 2c(\Gl) \big( - \sin{\Gt} \p_n u_1(x_0) + \cos{\Gt} \p_n u_2(x_0) \big)
  \big( \cos{\Gt} \nabla_n u_1(x_0) + \sin{\Gt} \nabla_n u_2(x_0) \big) \\
  &=0
\end{align*}
for any $\Gt \in \Rbb$.
Choosing $\Gt = 0$, we have
\begin{equation} \label{theta=45}
  \nabla_\p u_1(x_0) \cdot \nabla_\p u_2(x_0)  - c(\Gl) \p_n u_1(x_0) \cdot \p_n u_2(x_0) = 0.
\end{equation}
It then follows from \eqref{1=2} and \eqref{eq:1=-2} that
$$
  |\nabla_\p u_1(x_0)|^2 + c(\Gl) |\p_n u_1 |_-(x_0)|^2  = 0.
$$
Since $c(\Gl) >0$, we have $\p_n u_1 |_-(x_0)  = 0$, which contradicts \eqref{eq:nonzero}. Thus, $\GL(\Gt)$ is not constant in $\Gt$. Choose $\Gt_0$ such that $\GL(\Gt_0) \neq \GL(0)$, namely, $I_{(\cos{\Gt_0}) \Gvf_1 + (\sin{\Gt_0}) \Gvf_2}(x_0) \neq I_{\Gvf_1}(x_0)$. By replacing $\Gvf_2$ with $(\cos{\Gt_0}) \Gvf_1 + (\sin{\Gt_0}) \Gvf_2$, we achieve \eqref{eq:not=2}.
\end{proof}

Let $F \in \Mcal$. We denote eigenvalues of $\Kcal_F^*$ by $\Gl_j^{\pm,F}$ counting multiplicities where $+$ means positive eigenvalues and $-$ negative ones. They are enumerated in such a way that
\begin{equation*}
  \frac{1}{2}  > |\Gl_1^{\pm,F}| \ge |\Gl_2^{\pm,F}| \ge \cdots \to 0.
\end{equation*}
We mentioned that $1/2$ is a simple eigenvalue. We also mention that $\Kcal_F^*$ may have only finitely many negative eigenvalues. For example, if $D_F$ is strictly convex, it is the case as proved in \cite{MR19}.

For a positive integer $k$, let
\begin{equation}
  \Mcal_{k}^\pm := \{ F \in \Mcal \mid \Gl_i^{\pm,F}\text{ is simple for } 1 \le i \le k \}.
\end{equation}
Then we have
\begin{equation}\label{3000}
  \Mcal_s = \bigcap_{k=1}^\infty (\Mcal_{k}^+ \cap \Mcal_{k}^-).
\end{equation}

We are now ready to prove Theorem \ref{thm: generic simlicity}

\begin{proof}[Proof of Theorem \ref{thm: generic simlicity}]
We shall show that each $\Mcal_{k}^\pm$ is open and dense in the complete metric space $\ol{\Mcal}$. It follows from Baire's theorem (see, for example, \cite{Oxtoby}) and \eqref{3000} that $\Mcal_s$ is Baire typical in $\Mcal$ as desired.

We only prove $\Mcal_{k}^+$ is open and dense in $\ol{\Mcal}$ since the case for $\Mcal_{k}^-$ can be dealt with in the same way. To prove that $\Mcal_{k}^+$ is open, let $F \in \Mcal_{k}^+$. Choose $b>0$ so that $\Gl_k^{+,F} > b > \Gl_{k+1}^{+,F}$. If $\Gr(F,G)< \Ge$ for a sufficiently small $\Ge$, then the number of eigenvalues of $\Kcal_G^*$ in $(b,1/2)$ is $k$. Moreover, like the proof of Lemma \ref{eq:number}, we use by Lemma \ref{lem:ah} (i) and Theorem \ref{thm:G} to show that they are all different. Thus $G \in \Mcal_{k}^+$.

We now show that $\Mcal_{k+1}^+$ is dense in $\Mcal_{k}^+$ for $k=0,1,\ldots$ with $\Mcal_{0}^+=\Mcal$. Put $\Gl_0^+=1/2$. Suppose that $F \in \Mcal_{k}^+ \setminus \Mcal_{k+1}^+$. Then there is integer $N \ge 2$ such that
\begin{equation}\label{1000}
\Gl_{k}^{+,F} > \Gl_{k+1,}^{+,F} = \ldots = \Gl_{k+N}^{+,F} > \Gl_{k+N+1}^{+,F} .
\end{equation}

Let $\Gl:=\Gl_{k+1}^{+,F} = \ldots = \Gl_{k+N}^{+,F}$. Choose $b$ so that
$$
\Gl_{k+N}^{+,F} > b > \Gl_{k+N+1}^{+,F}.
$$

Suppose that $N=2$. Then $\Gl$ is an eigenvalue of multiplicity $2$. By Lemma \ref{eq:not=}, there are two normalized eigenfunctions $\Gvf^{(1)}, \Gvf^{(2)}$ corresponding to $\Gl$ satisfying
$$
I_{\Gvf_1} (x) \neq I_{\Gvf_2} (x) \quad \text{for some $x \in \p D_F$}.
$$
Choose $a \in C^\infty(\p D_F)$ with $\max_{x \in \p D_F} |a(x)| \le 1$ such that
$$
\int_{\p \GO} a I_{\Gvf^{(1)}} d\sigma \neq \int_{\p \GO} a I_{\Gvf^{(2)}} d\sigma.
$$
Then, by Theorem \ref{thm:G}, there are analytic branches, say $\Gl^{(k+1)}(h)$ and $\Gl^{(k+2)}(h)$, such that $\Gl^{(k+1)}(0)=\Gl^{(k+2)}(0)=\Gl$ and
\beq\label{eq:onetwoneq}
\frac{d \Gl^{(k+1)}}{dh} (0) \neq \frac{d \Gl^{(k+2)}}{dh} (0).
\eeq
Then $\Gl^{(k+1)}(\Ge) \neq \Gl^{(k+2)}(\Ge)$ for all sufficiently small $\Ge$. We may assume that $\Gl^{(k+1)}(\Ge) > \Gl^{(k+2)}(\Ge)$ without loss of generality. Let $\Gl^{(i)}(\Ge)$ be the analytic functions given in Theorem \ref{thm:G} such that $\Gl^{(i)}(0)= \Gl_i^{+,F}$, $1 \le i \le k$. Since $\Gl^{(1)}(0) > \Gl^{(2)}(0) > \ldots > \Gl^{(k)}(0) > \Gl$, we have
$$
\Gl^{(1)}(\Ge) > \Gl^{(2)}(\Ge) > \ldots > \Gl^{(k)}(\Ge) > \Gl^{(k+1)}(\Ge) > \Gl^{(k+2)}(\Ge)
$$
for all sufficiently small $\Ge$.
By Lemma \ref{lem:ah} (ii), there is $G \in \Mcal$ with $\Gr(F,G) < C\Ge$ such that $\p D_G= \p D_F(\Ge a)$.
Moreover, the number of eigenvalues of $\Kcal_G^*$ lying in $(b, 1/2)$ is $k+2$. Thus we have
$$
\Gl_1^{+,G} > \Gl_2^{+,G} > \ldots > \Gl_{k+2}^{+,G} .
$$
In particular, $G \in \Mcal_{k+1}$ and $\Gr(F,G) < C\Ge$. Thus $\Mcal_{k+1}^+$ is dense in $\Mcal_{k}^+$ if $N=2$.

Suppose that $N \ge 3$. We choose two normalized eigenfunctions $\Gvf^{(1)}, \Gvf^{(2)}$ in $E$ satisfying \eqref{eq:not=}. Extend them to a basis of the eigenspace corresponding to $\Gl$ and denote them by $\Gvf^{(1)}, \Gvf^{(2)}, \ldots, \Gvf^{(N)}$. By the same argument as above, we see that there $G \in \Mcal$ with $\Gr(F,G) < C\Ge$ such that
$$
\Gl_1^{+,G} > \Gl_2^{+,G} > \ldots > \Gl_{k+1}^{+,G} \ge \Gl_{k+2}^{+,G} \ge \ldots \ge \Gl_{k+N}^{+,G}
$$
and at most $N-1$ of $\Gl_{k+1}^{+,G} , \ldots , \Gl_{k+N}^{+,G}$ can be identical. If $\Gl_{k+1}^{+,G}$ is simple, then we are done. If not, then $\Gl_{k+1}^{+,G} = \ldots = \Gl_{k+M}^{+,G}$ for some $2 \le M < N$. We then repeat the same arguments until $M=2$. This completes the proof.
\end{proof}

\section{Proof of Theorems \ref{cyclic2} and \ref{cyclic3}}\label{sec:three}

We first prove the following simple lemma.

\begin{lemma}\label{lem:M<N}
Let $H$ be a separable Hilbert space and let $K$ be a self-adjoint compact operator on $H$. Suppose that the maximal multiplicity of eigenvalues of $K$ is $N < +\infty$. If
$$
L_{\xi_1} + L_{\xi_2} + \cdots + L_{\xi_M} =H,
$$
then $M \ge N$.
\end{lemma}

\begin{proof}
Let $\xi_1, \xi_2, \ldots, \xi_M \in H$ with $M<N$. Let $E$ be an eigenspace corresponding to an eigenvalue $\Gl$ such that $\dim E =N$. Let $P$ be the orthogonal projection on $E$. Then, $\dim P (\mbox{span}\{\xi_1, \xi_2, \ldots, \xi_M \}) \leq M <N$. Thus
there exists $\phi \in E$ such that $\phi \perp P (\mbox{span}\{\xi_1, \xi_2, \ldots, \xi_M \})$. Since $P \phi =\phi$, $\phi \perp \xi_l$ for all $l$.
It then follows that
$$
\langle K^{k} \xi_l, \phi \rangle = \Gl_j^k \langle \xi_l, \phi \rangle =0 \quad \forall k \in {\mathbb N}, \forall l=1, 2, \ldots, M,
$$
and hence
$$
\phi \in (L_{\xi_1} + \cdots + L_{\xi_M})^{\perp}.
$$
So $L_{\xi_1} + \cdots + L_{\xi_M} \neq H$.
\end{proof}

The following theorem characterizes those $N$-cyclic vectors $\xi_1, \xi_2, \ldots, \xi_N \in H$.

\begin{theorem}\label{cyclic}
Let $H$ be a separable Hilbert space and let $K$ be a self-adjoint compact operator on $H$. Suppose that the maximal multiplicity of the eigenvalues of $K$ is $N < +\infty$. Then, \eqref{fulltwo} holds if and only if for any eigenvalue $\Gl$ of multiplicity $N_\Gl$ the $N_\Gl \times N$ matrix $A_\Gl := \left( \langle \Gvf_j, \xi_k \rangle \right)$ satisfies
\begin{equation}\label{rank}
\rank A_\Gl =N_\Gl,
\end{equation}
where $\{\Gvf_1, \Gvf_2, \ldots, \Gvf_{N_\Gl} \}$ is an orthonormal basis of the eigenspace of $\Gl$.
\end{theorem}

If $N=1$, namely, all the eigenvalues are simple, then the condition \eqref{rank} amounts to
\beq
\la \xi, \Gvf \ra \neq 0 \quad \mbox{for any eigenfunction $\Gvf$}.
\eeq
Thanks to \eqref{rank}, it is now easy to construct $N$-cyclic vectors. Suppose that the maximal multiplicity of the eigenvalues of a self-adjoint compact operator $K$ is $N < +\infty$. Let $\{ \Gl_1, \Gl_2, \ldots \}$ be the eigenvalues of $K$ and $N_j$ be the multiplicity of $\Gl_j$. Then $N_j \le N$. Let $\{\Gvf_{j,1}, \Gvf_{j,2}, \ldots, \Gvf_{j,N} \}$ be such that the first $N_j$ of them are orthonormal vectors spanning the corresponding eigenspace and the rest are zero vectors. Then, define for $k=1,2,\ldots, N$
$$
\xi_k := \sum_{j=1}^\infty c_{j,k} \Gvf_{j,k}
$$
where $c_{j,k} \neq 0$ for any $j,k$. These vectors $\xi_1, \xi_2, \ldots, \xi_N$ constitute $N$-cyclic vectors for $K$.

\begin{proof}[Proof of Theorem \ref{cyclic}]
Suppose that \eqref{fulltwo} holds. Let $\Gl$ be an eigenvalue of $K$ and let $\{\Gvf_1, \Gvf_2, \ldots, \Gvf_{N_\Gl} \}$ be an orthonormal basis of the eigenspace of $\Gl$. Then, $\Gvf_j \in L_{\xi_1} + \cdots + L_{\xi_N}$ for $j=1, 2, \ldots, N_\Gl$. Thus, for any $\Ge>0$, there are polynomials $p_k^j$, $1 \le k \le N, \ 1 \le j \le N_\Gl$, such that $\psi_j :=\sum_{k=1}^N p_k^j(K) \xi_k$ satisfies $\| \psi_j - \Gvf_j \| < \Ge$. We then have $\la \psi_j, \Gvf_j \ra > 1-\Ge$. Moreover, since $\la \Gvf_j, \Gvf_{l} \ra =0$ if $j \neq l$, we have $|\la \psi_j, \Gvf_{l} \ra| < \Ge $ if $j \neq l$. We have
$$
\sum_{k=1}^N \langle \Gvf_j, \xi_k \rangle p_k^l(\Gl) = \sum_{k=1}^N \langle p_k^l(K) \Gvf_j, \xi_k \rangle = \left\la \Gvf_j, \psi_l  \right\ra , \quad j, l=1, 2, \ldots, N.
$$
So, if we set $B$ be the $N \times N_\Gl$ matrix $B=(p_k^l(\Gl))$, the diagonal entries of $A_\Gl B$ are larger than $1-\Ge$ in their absolute values, off-diagonal ones are less than $\Ge$. So, if $\Ge$ is sufficiently small, then $A_\Gl B$ is invertible, and hence \eqref{rank} holds.

Conversely, assume that \eqref{fulltwo} does not hold. Then there is a normalized eigenfunction $\Gvf$ of $K$ such that $\Gvf \in (L_{\xi_1} + \cdots + L_{\xi_N})^\perp$. Let $\Gl$ be the eigenvalue corresponding to $\Gvf$ and $\{\Gvf_1, \Gvf_2, \ldots, \Gvf_{N_\Gl} \}$ be an orthonormal basis of the eigenspace of $\Gl$ with $\Gvf_1=\Gvf$.  Since $\Gvf_1 \in L_{\xi_k}^\perp$ for all $k$, we have $\langle \xi_k, \Gvf \rangle =0$ for all $k$. Thus, we have
$$
\rank \left( \langle \Gvf_j, \xi_k \rangle \right) \le N_\Gl -1.
$$
This completes the proof.
\end{proof}

We now prove Theorem \ref{cyclic2}.

\begin{proof}[Proof of Theorem \ref{cyclic2}]
We assume $d=3$ since the same proof works for the case $d \ge 2$. Let $\Gl$ be an eigenvalue of $\Kcal_{\p\GO}^*$ on $H^{-1/2}(\p\GO)$, and let $N_\Gl$ be the multiplicity of $\Gl$ and $\Gvf_1, \Gvf_2, \ldots, \Gvf_{N_\Gl}$ be corresponding orthonormal eigenfunctions.  Let $\psi_j = \Scal_{\p\GO} [\Gvf_j]$. Then, we have
\beq\label{Fj}
F_j(z):= \langle q_z, \Gvf_j \rangle_{\p\GO} = v \cdot \nabla \int_{\p\GO} \Gamma(z-x) \psi_j(x) d\sigma(x),
\eeq
namely,
$$
F_j(z)= v \cdot \nabla \Scal_{\p\GO} [\psi_j](z), \quad z \in \Rbb^3 \setminus \ol{\GO}.
$$
Note that $F_j$ is harmonic in $\Rbb^d \setminus \ol{\GO}$. We emphasize that $F_j$ is not identically zero. In fact, if $F_j \equiv 0$ for some constant vector $v$, we may assume $v=(1,0,0)$ by rotation if necessary. There is a harmonic function $u$ of two variables such that
$$
\Scal_{\p\GO} [\psi_j](z_1,z_2,z_3)= u(z_2,z_3)
$$
for all $z= (z_1,z_2,z_3)$. Since $\Scal_{\p\GO} [\psi_j](z) \to 0$ as $|z| \to \infty$, we infer $u(z_2,z_3)=0$ by sending $z_1$ to $\infty$. Thus $\Scal_{\p\GO} [\psi_j](z)=0$ for all $z \in \Rbb^3 \setminus \ol{\GO}$. In particular, $\Scal_{\p\GO} [\psi_j] \equiv 0$ on $\p\GO$. Since $\Scal_{\p\GO} [\psi_j] \in H^{1/2}(\p\GO)$, it follows from the uniqueness of the solution to Dirichlet's problem that $\Scal_{\p\GO} [\psi_j](z)=0$ for all $z \in \GO$. By the jump formula \eqref{singlejump}, we have
$$
\psi_j = \p_n \Scal_{\p\GO} [\psi_j] |_+ - \p_n \Scal_{\p\GO} [\psi_j] |_- \quad\text{on } \p\GO.
$$
Thus $\psi_j=0$ which is absurd. So, $F_j$ is not identically zero.

Choose $N$ different points $z_1, \ldots, z_N \notin \ol{\GO}$ and let $\xi_k:= q_{z_k}$. Then, the matrix $A_\Gl= (\langle \xi_k, \Gvf_j \rangle)$ is given by
$$
A_\Gl = ( F_j(z_k) ).
$$
Let $A_{\Gl,m}$, $m=1, 2, \ldots, M_\Gl:= \begin{pmatrix} N \\ N_\Gl \end{pmatrix}$, be $N_\Gl \times N_\Gl$ submatrices of $A_\Gl$. Let
$$
G_\Gl(z_1, \ldots, z_N):= \sum_{m=1}^{M_\Gl} (\det A_{\Gl,m})^2.
$$
Note that $\rank A_\Gl < N_\Gl$ if and only if $G_j(z_1, \ldots, z_N)=0$.

Let
$$
Z_\Gl := \{ (z_1, \ldots, z_N) \in (\Rbb^3 \setminus \ol{\GO})^N \mid G_\Gl(z_1, \ldots, z_N)=0 \}.
$$
Then, $\rank A_\Gl = N_\Gl$ for any eigenvalue $\Gl$ of $\Kcal_{\p\GO}^*$ if $(z_1, \ldots, z_N) \notin \cup Z_\Gl$ where the union is taken over all eigenvalues $\Gl$ of $\Kcal_{\p\GO}^*$.  Since $F_j$ is harmonic in $\Rbb^3 \setminus \ol{\GO}$, $G_\Gl$ is real analytic and hence $Z_\Gl$ is of Lebesgue measure zero, and so is $\cup Z_\Gl$. Thus \eqref{full3} follows from Theorem \ref{cyclic}.
\end{proof}

\begin{proof}[Proof of Theorem \ref{cyclic3}]
Let $\{ \Gvf_1, \Gvf_2, \ldots \}$ be an orthonormal system of eigenfunctions for $H^{-1/2}(\p\GO)$. Let $Z_j$ be the zero set of the function $F_j$ defined by \eqref{Fj}. Then each $Z_j$ is of measure zero and for any $z \notin \cup_{j=1}^\infty Z_j$, \eqref{eigenmode} holds.
\end{proof}

\section*{Discussion}
It would be interesting to know whether the injectivity of the NP operator is also generic among all smooth, closed hypersurfaces in Euclidean space. There are hypersurfaces where $0$ is an NP eigenvalue. For example, $0$ is an NP eigenvalue of infinite multiplicities on lemniscates in two-dimensions as mentioned in Introduction. Some oblate spheroids have $0$ as its NP eigenvalue \cite{Ritter-95} (see also \cite{FK}).


\end{document}